\documentclass[12pt]{amsart}

\usepackage{amssymb}
\usepackage{amsmath}

\def\nn{{\mathbb N}}
\def\qq{{\mathbb Q}}

\def\zz{{\mathbb Z}}
\def\pp{{\mathbb P}}

\def\ov{\overline}

\def\fx{{\rm Fix}}
\def\Des{{\rm Des}}
\def\Exc{{\rm Exc}}
\def\maj{{\rm maj}}
\def\exc{{\rm exc}}
\def\dex{{\rm Dex}}
\def\des{{\rm des}}
\def\fix{{\rm fix}}

\def\pr{{\rm Par}}

\def\p{{\sf p}}
\def\h{{\sf h}}

\def\H{{\sf H}}
\def\LL{{\sf L}}

\def\rc{{\rm C}}

\def\cm{S_\mu} 
\def\cjl{S_{\lambda,j}}
\def\ajl{a_{\lambda,j}}

\def\qsym{{\mathcal Q}}

\newtheorem{thm}{Theorem}[section]
\newtheorem{lemma}[thm]{Lemma}
\newtheorem{cor}[thm]{Corollary}
\newtheorem{prop}[thm]{Proposition}

\newtheorem{notation}[thm]{Notation}

\theoremstyle{definition}

\theoremstyle{remark}
\newtheorem{remark}[thm]{Remark}

\begin{document}

\title{Eulerian quasisymmetric functions and cyclic sieving}         
\author[Sagan]{Bruce Sagan} 
\address{Department of Mathematics, Michigan State University, East Lansing, MI 48824}
\email{sagan@math.msu.edu}

\author[Shareshian]{John Shareshian$^1$}
\address{Department of Mathematics, Washington University, St. Louis, MO 63130}
\thanks{$^{1}$Supported in part by NSF Grants
DMS 0604233 and 0902142}
\email{shareshi@math.wustl.edu}

\author[Wachs]{Michelle L. Wachs$^2$}
\address{Department of Mathematics, University of Miami, Coral Gables, FL 33124}
\email{wachs@math.miami.edu}
\thanks{$^{2}$Supported in part by NSF Grants
DMS 0604562 and 0902323}

\date{September 15, 2009}          

\maketitle

 \vspace{-.1in}\begin{center}{\it Dedicated to Dennis Stanton}
\end{center}

\begin{abstract} It is shown that a refined version of a q-analogue of the Eulerian numbers together with the  action, by conjugation, of the  subgroup of the symmetric group $S_n$ generated   by the $n$-cycle $(1,2,\dots,n)$ on the set of permutations of fixed cycle type and fixed number of excedances provides an instance of the cyclic sieving phenonmenon of  Reiner, Stanton and White.  The main tool is a class of  symmetric functions  recently introduced in work  of two of the authors.
 \end{abstract}

\
\section{Introduction}  \label{intro}     

In \cite{sw,ShWa1}, certain quasisymmetric functions, called ``Eulerian quasisymmetric functions" are introduced and shown to be in fact symmetric functions.  These symmetric functions have been useful in the study of the joint distribution of the  permutation statistics, major index and excedance number.   There are various versions of the Eulerian quasisymmetric functions. They are defined  by first associating a fundamental quasisymmetric function with each permutation in the symmetric group $S_n$ and then summing these fundamental quasisymmetric functions over permutations in certain subsets of $S_n$.     To obtain the most refined versions,  the cycle-type Eulerian quasisymmetric functions $Q_{\lambda,j}$,   one sums the fundamental quasisymmetric functions associated with the permutations  having exactly $j$ excedances and cycle type $\lambda$.   By summing over all the permutations in  $S_n$ having $j$ excedances and $k$ fixed points, one obtains the less refined version $Q_{n,j,k}$.  The precise  definition of the Eulerian quasisymmetric functions and other  key terms can be found in Section \ref{defs}.

Shareshian and Wachs \cite{sw,ShWa1} derive a formula for the generating function of $Q_{n,j,k}$ which specializes to a  $(q,r)$-analog of a classical formula for the exponential generating function of the Eulerian polynomials.  The $(q,r)$-analogue  of the  classical formula is given by
 \begin{equation} \label{expgeneqfix}
\sum_{n \geq 0}A^{\maj,\exc,\fix}_n(q,t,r)\frac{z^n}{[n]_q!}=\frac{(1-tq)\exp_q(rz)}{\exp_q(ztq)-tq\exp_q(z)},
\end{equation}
where 
 $$A^{\maj,\exc,\fix}_n(q,t,r) := \sum_{\sigma \in S_n} q^{\maj(\sigma)} t^{\exc(\sigma)} r^{\fix(\sigma)},$$  and 
\[
\exp_q(z):=\sum_{n \geq 0}\frac{z^n}{[n]_q!}.
\]

The cycle-type Eulerian quasisymmetric functions $Q_{\lambda,j}$   remain somewhat mysterious, and one might expect that better understanding of them will lead to further results on permutation statistics.  In this paper, we provide evidence that this expectation is reasonable.  With Theorem \ref{psum}, we prove a conjecture from \cite{ShWa1}, describing the expansion of $Q_{\lambda,j}$, for  $\lambda=(n)$,  in terms of the power sum basis for the space of symmetric functions.  Combining Theorem~\ref{psum} with a technique of D\'esarm\'enien \cite{des},  we are able to evaluate, at all $n^{th}$ roots of unity,
 the {\it cycle-type $q$-Eulerian numbers} 
\begin{equation} \label{defqeuler}
\ajl(q):=\sum_{\sigma \in \cjl}q^{\maj(\sigma)-\exc(\sigma)},
\end{equation}
where $\cjl$ is the set of all $\sigma \in S_n$ having exactly $j$ excedances and cycle type $\lambda$.  This  and an analysis of the excedance statistic on the centralizers $C_{S_n}(\tau)$  of  certain permutations $\tau\in S_n$ enable us to establish the 
 relationship between 
    the polynomials $a_{\lambda,j}(q)$
and the  cyclic sieving phenomenon of Reiner, Stanton and White \cite{RSW} given in Theorem \ref{main1} below.   

\begin{notation}
For a positive integer $d$, $\omega_d$ will denote throughout this paper an arbitrary complex primitive $d^{th}$ root of $1$.
\end{notation}

\begin{thm} \label{main1}
Let $\gamma_n=(1,2,\ldots, n) \in S_n$ and let $G_n=\langle \gamma_n \rangle \leq S_n$.  Then for all  partitions  $\lambda$  of $n$ and  $j \in \{0,1,\dots,n-1\}$,
the group $G_n$ acts on $\cjl$ by conjugation and the triple $(G_n,\cjl,\ajl(q))$ exhibits the cyclic sieving phenomenon.  In other words, if $\tau \in G_n$ has order $d$ then
\begin{equation} \label{maineq1}
\ajl(\omega_d)=|C_{S_n}(\tau) \cap \cjl|.
\end{equation}
\end{thm}

For $\lambda$ a partition of $n$, let $S_\lambda$ be the set of all $\sigma \in S_n$ of cycle type $\lambda$ and define the {\it cycle type Eulerian polynomial} associated with $\lambda$ as $$A_\lambda^{\maj,\exc}(q,t):= \sum_{\sigma \in S_\lambda} q^{\maj(\sigma)} t^{\exc(\sigma)}.$$ Then 
 (\ref{maineq1}) can be rewritten as
\begin{equation} \label{maineq2} A_\lambda^{\maj,\exc}(\omega_d, t \omega_d^{-1}) = \sum_{\sigma \in C_{S_n}(\tau) \cap S_\lambda} t^{\exc(\sigma)},\end{equation}  
 which is clearly a refinement of 
\begin{equation} \label{majexcfixeq} A_n^{\maj,\exc,\fix}(\omega_d,t\omega_d^{-1},s) = \sum_{\sigma \in C_{S_n}(\tau) } t^{\exc(\sigma)} s^{\fix(\sigma)}.\end{equation} 
We also prove that both sides of (\ref{majexcfixeq}) are, in fact,  equal to 
\begin{equation} \label{bothsideseq} A^{\maj,\exc,\fix}_{\frac n d}(1,t, \frac{s^d+t[d-1]_t}{[d]_t})\,\, [d]_t^{\frac n d},\end{equation} which by setting $s=1$ yields
$$ A_{n}^{\maj,\exc}(\omega_d,t \omega_d^{-1}) = A_{\frac{n}{d}}(t) \,[d]^{\frac {n} d}_t,$$
for all divisors $d$ of $n$.
For the  cycle-type Eulerian polynomial $A_{(n+1)}^{\maj,\exc}(q,t)$, we  obtain the similar looking result,
\begin{equation} \label{otherform} A_{(n+1)}^{\maj,\exc}(\omega_d,t \omega_d^{-1}) = t A_{\frac{n}{d}}(t) \,[d]^{\frac {n} d}_t,
\end{equation}
for all divisors $d$ of $n$.

The paper is organized as follows.  In Section~\ref{defs}, we review  definitions of various terms such as  cyclic sieving and Eulerian quasisymmetric functions.  We also present some preliminary results on Eulerian quasisymmetric functions from \cite{ShWa1}.    In Section~\ref{symsec} we describe a technique that uses symmetric function theory  to evaluate certain polynomials at roots of unity  based on  work of D\'esarm\'enien \cite{des}.   
Theorem~\ref{psum} mentioned above is proved in Section~\ref{psumsec} by means of results of \cite{ShWa1} which enable one to express the cycle-type Eulerian quasisymmetric functions in terms of the less refined version of Eulerian quasisymmetric functions.   The proof of Theorem~\ref{main1} appears in Section~\ref{eqth1sec}.
 In  Section~\ref{finsec}, we  prove that both sides of (\ref{majexcfixeq}) are equal
to (\ref{bothsideseq}), that (\ref{otherform}) holds, and  that another  triple  exhibts the cyclic sieving phenomenon, namely $(G_{n}, S_{n,j}, a_{(n+1),j+1}(q))$, where $S_{n,j}$ is the set of all permutations in $S_n$ with $j$ excedances.

\section{Definitions, known facts and preliminary results} \label{defs}

\subsection{Cyclic Sieving} \label{cycsiv}

 Let $G$ be a finite cyclic group acting on a set $X$, and let $f(q)$ be a polynomial in $q$ with nonnegative integer coefficents.  
For $g \in G$, let $\fx(g)$ be the set of fixed points of $g$ in $X$.   The triple $(G,X,f(q))$ {\it exhibits the cyclic sieving phenomenon}  of  Reiner, Stanton and White  \cite{RSW} if for each $g \in G$ we have
\begin{equation} \label{cseq}
f(\omega_{|g|})=|\fx(g)|,
\end{equation}
where $|g|$ is the order of $g$.

\begin{remark}  
Since all elements of order $d$ in a cyclic group $G$ generate the same subgroup, they have the same set of fixed points in any action.  Thus our formulation of the cyclic sieving phenomenon  is equivalent to the definition given in \cite{RSW}.
\end{remark}

Note that if $(G,X,f(q))$ exhibits the cyclic sieving phenomenon then $f(1)=|X|$, and interesting examples arise where $f(q)$ is the generating function for some natural statistic on $X$, that is, there exists some useful function $s:X \rightarrow \nn$ such that
\[
f(q)=\sum_{x \in X}q^{s(x)}.
\]
See \cite{RSW} for many examples, and \cite{RSW2,EF,BRS, BR, Rho, PPR, PS} for more recent work.

\subsection{Permutation Statistics}
Recall that for a permutation $\sigma \in S_n$ acting from the right on $[n]:=\{1,\ldots,n\}$, the {\it excedance set} of $\sigma$ is
\[
\Exc(\sigma):=\{i \in [n-1]:i\sigma>i\}
\]
and the {\it descent set} of $\sigma$ is
\[
\Des(\sigma):=\{i \in [n-1]:i\sigma>(i+1)\sigma\}.
\]
The {\it major index} of $\sigma$ is
\[
\maj(\sigma):=\sum_{i \in \Des(\sigma)}i,
\]
and the {\it excedance and descent statistics} of $\sigma$ are, respectively,
\[
\exc(\sigma):=|\Exc(\sigma)|,
\]
and
\[
\des(\sigma):=|\Des(\sigma)|.
\]
Let $\fx(\sigma)$  denote the set of fixed points of $\sigma$, that is 
$$\fx(\sigma) := \{i \in [n] : i\sigma = i \}$$
and let 
$$\fix(\sigma) := |\fx(\sigma)|.$$

The excedance and descent statistics are equidistributed,  and the $n^{th}$ {\it Eulerian polynomial} $A_n(t)$ can be defined as
\[
\sum_{\sigma \in S_n}t^{\exc(\sigma)}=A_n(t)=\sum_{\sigma \in S_n}t^{\des(\sigma)}.
\]
The Eulerian polynomial is also the generating polynomial for the {\it ascent statistic} on $S_n$, ${\rm asc}(\sigma):=|\{i \in [n]:i\sigma  < (i+1)\sigma\}|$.

For permutation statistics $
{s_1},\ldots,{ {s_k}}$ and a positive integer $n$, define
the polynomial
\[
A_n^{ s_1,\ldots,s_k}(t_1,\ldots,t_k):=\sum_{\sigma \in
S_n}t_1^{s_1  (\sigma) } t_2^{ s_2
(\sigma)}\cdots t_k^{s_k  (\sigma)}.
\]
Also, set
\[
A_0^{ s_1,\ldots,s_k}(t_1,\ldots,t_k):=1.
\]

\subsection{Partitions and Symmetric Functions}

We use standard notation for partitions and symmetric functions.  References for basic facts are \cite{Mac,Sta2,Sag}.  In particular,  $\p_\lambda$ and $\h_\lambda$ will denote, respectively, the power sum and complete homogeneous symmetric functions associated to a partition $\lambda$. We use $l(\lambda)$ to denote the number of (nonzero) parts of $\lambda$ and $m_j(\lambda)$ to denote the number of parts of $\lambda$ equal to $j$.   We write $\pr(n)$ for the set of all partitions of $n$.  For $\lambda \in \pr(n)$, define the number
$$
z_\lambda:=\prod_{j=1}^{n}j^{m_j(\lambda)}m_j(\lambda)!.$$

We use two standard methods to describe a partition $\lambda \in \pr(n)$.  The first is to write $\lambda=(\lambda_1,\ldots,\lambda_{l(\lambda)})$, listing the (nonzero) parts of $\lambda$ so that $\lambda_i \geq \lambda_{i+1}$ for all $i$.  The second is to write $\lambda=1^{m_1(\lambda)}\ldots n^{m_n(\lambda)}$, usually suppressing those synbols $i^{m_i(\lambda)}$ such that $m_i(\lambda)=0$ and writing $i^1$ as simply $i$.  In particular, if $n=dk$ then $d^k$ represents the partition with $k$ parts of size $d$ and no other parts. 
If $\lambda=(\lambda_1,\ldots,\lambda_k) \in \pr(n)$ and $q \in \qq$ with $q\lambda_i \in \pp$ for all $i \in [k]$, we write $q\lambda$ for $(q\lambda_1,\ldots,q\lambda_k) \in \pr(qn)$.

For each $\sigma \in S_n$, let $\lambda(\sigma)$ denote the cycle type of $\sigma$.
Given $\lambda \in \pr(n)$, we write $S_\lambda$ for the set of all $\sigma \in S_n$ having cycle type $\lambda$.  As in (\ref{defqeuler}), we write $\cjl$ for the set of those $\sigma \in S_\lambda$ satisfying $\exc(\sigma)=j$.

For symmetric functions $f,g$ with coefficients in $\qq[t]$, $f[g]$ will denote the plethysm of $g$ by $f$.  The same notation will be used for plethysm of symmetric power series with no bound on their degree.  One such power series is $\H:=\sum_{n \geq 1}\h_n$.  If we set
\begin{eqnarray*}
\LL & := & \sum_{d \geq 1} \frac{\mu(d)}{d}\log(1+\p_d) \\ & = & \sum_{d \geq 1}\frac{\mu(d)}{d}\sum_{i \geq 1}\frac{(-1)^{i-1}}{i}\p_d^i,
\end{eqnarray*}
where $\mu$ is the classical M\"obius function, then $\H$ and $\LL$ are plethystic inverses, that is,
\begin{equation} \label{plinv}
\LL[\H[f]]=\H[\LL[f]]=f
\end{equation}
for all symmetric power series $f$.  (This is due to Cadogan, see \cite{Cad} or \cite[Exercise 7.88e]{Sta2}.)  Note also that for any power series $h(t,x_1,x_2,\ldots)$ with coefficients in $\qq$ that is symmetric in $x_1,x_2,\ldots$ and any $d \in \pp$, we have
\begin{equation} \label{pple}
\p_d[h]=h(t^d,x_1^d,x_2^d,\ldots).
\end{equation}
We shall use without further mention the facts $(f+g)[h]=f[h]+g[h]$ and $(fg)[h]=f[h]g[h]$.

\subsection{$q$-Analogues}
We use the standard notation for polynomial analogues of positive integers, that is, for a positive integer $n$ and a variable $q$, we define
\[
[n]_q:=\sum_{j=0}^{n-1}q^j = \frac{1-q^n}{1-q}
\]
and
$$[n]_q!:= [n]_q[n-1]_q\cdots[1]_q.$$
Also define $$[0]_q! :=1.$$
It is well-known  that for any sequence $(k_1,\dots, k_m)$ of nonnegative integers whose sum is  $n$, the  $q$-multinomial coefficient
 $$ \left [ \begin{array}{c} n \\ k_1,\dots ,k_{m} \end{array} \right] _q := \frac{[n]_q!} {[k_1]_q! \cdots [k_m]_q! } $$
is always a polynomial in $\nn[q]$.  The following $q$-analogue of the multinomial version of  the Pascal recurrence relation is also well-known (see \cite[(17b)]{Sta1}):
\begin{equation} \label{pascal}  \left [ \begin{array}{c} n \\ k_1,\dots ,k_m \end{array} \right]_q =    \sum_{i=1}^{m} q^{ k_{i+1} +\dots + k_{m}}
 \left [ \begin{array}{c} n-1 \\ k_1,\dots ,k_i-1, \dots, k_{m} \end{array} \right]_q,\end{equation}
 where $(k_1,\dots, k_m)$ is a sequence of positive integers whose sum is $n$.
We will need the  following  elementary fact, which also plays a role in the work of Reiner, Stanton and White \cite{RSW}.

\begin{prop}[see {\cite[Equation (4.5)] {RSW}}]  \label{multinom}   Let   $(k_1,\dots,k_m)$  be  a sequence of nonnegative integers whose sum is $n$. If $d|n$ then
 $$ \left [ \begin{array}{c} n \\ k_1,\dots ,k_{m} \end{array} \right] _q |_{q=\omega_d} =
 \begin{cases}  \left ( \begin{array}{c} \frac n d \\ \frac {k_1} d,\dots ,\frac {k_{m}} d \end{array} \right) &\mbox{ if } d| k_i \,\, \forall i \in [m]
 \\  \,\, 0 &\mbox{ otherwise. } \end{cases}$$
\end{prop}

\subsection{The Eulerian quasisymmetric functions} \label{Eulerian}

Given a permutation $\sigma \in S_n$, we write $\sigma$ in one line notation,
\[
\sigma=\sigma_1\ldots\sigma_n,
\]
where $\sigma_i=i\sigma$.  Set
\[
[\ov{n}]:=\{\ov{i}:i \in [n]\}, 
\]
and let $w(\sigma)$ be the word in the alphabet ${\mathcal A}:=[n] \cup [\ov{n}]$ obtained from $\sigma$ by replacing $\sigma_i$ with $\ov{\sigma_i}$ whenever $i \in \Exc(\sigma)$.  Order ${\mathcal A}$ by
\[
\ov{1}<\ldots<\ov{n}<1<\ldots<n,
\]
and for any word $w:=w_1\ldots w_n$ from ${\mathcal A}$, set
\[
\Des(w):=\{i \in [n-1]:w_i>w_{i+1}\}.
\]
Now, for $\sigma \in S_n$, define
\[
\dex(\sigma):=\Des(w(\sigma)).
\]
For example, if $\sigma=641532$ then $w(\sigma)=\ov{6}\ov{4}1\ov{5}32$ and $\dex(\sigma)=\{1,3,5\}$.

Recall now that a {\it quasisymmetric function} is a power series (with rational coefficients) $f$ of bounded degree in variables $x_1,x_2,\ldots$ such that if $j_1<\ldots<j_k$ and $l_1<\ldots<l_k$, then for all $a_1,\ldots,a_k$ the coefficients in $f$ of $\prod_{i=1}^{k}x_{j_i}^{a_i}$ and $\prod_{i=1}^{k}x_{l_i}^{a_i}$ are equal.  The usual addition, multiplication and scalar multiplication make the set $\qsym$ of quasisymmetric functions a $\qq$-algebra that strictly contains the algebra of symmetric functions. For $n \in \pp$ and $S \subseteq [n-1]$, set
\[
{\rm Mon}_S:=\{\prod_{i=1}^{n}x_{j_i}:j_i \geq j_{i+1} \mbox{ for all } i \in [n-1] \mbox{ and } j_i>j_{i+1} \mbox{ for all } i \in S\},
\]
and define the {\it fundamental quasisymmetric function} associated with $S$ to be
\[
F_S:=\sum_{x \in {\rm Mon}_S}x \in \qsym.
\]
Recall from above that we have defined $\cjl$ to be the set of all permutations of cycle type $\lambda$ with $j$ excedances.  The {\it Eulerian quasisymmetric function} associated to the pair $(\lambda,j)$ is
\[
Q_{\lambda,j}:=\sum_{\sigma \in \cjl}F_{\dex(\sigma)} \in \qsym.
\]

The Eulerian quasisymmetric functions were introduced in \cite{ShWa1} as a tool for studying the $(\maj,\exc)$ $q$-analogue of the Eulerian polynomials.  The connection between the Eulerian quasisymmetric functions and the $q$-Eulerian numbers is given in the following proposition. The {\em stable principal specialization} $\Omega$  is a homomorphism from the algebra of quasisymmetric functions $\qsym$ to the algebra of formal power series $\qq[[q]]$ defined by
$\Omega(x_i) = q^{i-1}$. 
 
 \begin{prop}[{\cite[Equation (2.13)] {ShWa1}}] {\footnote{Equation~2.13 in \cite{ShWa1} has an extra factor of $q^{j}$ because $a_{\lambda,j} (q)$ is defined there to be the $\maj$ enumerator of $\cjl$ rather than the   $\maj-\exc$ enumerator.}} \label{stabprop}For all partitions $\lambda$ of $n$ and $j\in\{0,1,\dots, n-1\}$, let $a_{\lambda,j}(q)$ be as in (\ref{defqeuler}).  Then
$$\Omega Q_{\lambda,j} = a_{\lambda,j} (q) / \prod_{i=1}^n (1-q^i). $$ 
 \end{prop}
 
In \cite{ShWa1}, it is also shown that in fact $Q_{\lambda,j}$ is always a symmetric function. 
 If one knows $Q_{(n),j}$ for all $n,j$, then a fairly compact  explicit formula for each $Q_{\lambda,j}$ can be obtained from  Corollary 6.1 of \cite{ShWa1}, which says that for any $\lambda \in \pr(n)$,
\begin{equation} \label{ple}
\sum_{j=0}^{n-1}Q_{\lambda,j}t^j=\prod_{i=1}^{n}\h_{m_i(\lambda)}\left[\sum_{l=0}^{i-1}Q_{(i),l}t^l\right].
\end{equation}

As noted in \cite{ShWa1}, if we set 
\[
Q_{n,j}:=\sum_{\lambda \in \pr(n)}Q_{\lambda,j}
\] for $n \ge 1$, and $$Q_{0,0}=Q_{(0),0}=1,$$ equation (\ref{ple}) implies
$$
\sum_{n,j \geq 0}Q_{n,j}t^j=\sum_{n \geq 0}\h_n\left[\sum_{i,j \geq 0}Q_{(i),j}t^j\right]=\H\left[\sum_{i,j \geq 0}Q_{(i),j}t^j\right],
$$
which by (\ref{plinv}) is equivalent to 
\begin{equation} \label{ple2}
\sum_{n,j \geq 0}Q_{(n),j}t^j=\LL\left[\sum_{i,j \geq 0}Q_{i,j}t^j\right].
\end{equation}
Proposition 6.6 of \cite{ShWa1} gives an explicit formula for $Q_{n,j}$ in terms of the power sum symmetric function basis,
\begin{equation} \label{qnj}
\sum_{j=0}^{n-1}Q_{n,j}t^j=\sum_{\nu \in \pr(n)}z_\nu^{-1}A_{l(\nu)}(t)\prod_{i=1}^{l(\nu)}[\nu_i]_t\p_\nu.
\end{equation}
By combining  (\ref{ple}), (\ref{ple2}) and (\ref{qnj}) we obtain a formula for each $Q_{\lambda,j}$, which will be used in Section~\ref{psumsec}  to prove a conjecture from \cite{ShWa1} giving the expansion of $Q_{(n),j}$ in the power sum basis.

\section{A symmetric function technique} \label{symsec}

We describe here a general technique for evaluating polynomials at roots of unity based on a technique of D\'esarm\'enien \cite{des}.  This technique provides a key step in our proof of Theorem~\ref{main1}.   One can also prove Theorem 1.1 using Springer's theory of regular elements in place of the technique we give here.  A description of the relevance of Springer's work to the cyclic sieving phenomenon appears in [RSWh].

 Given a homogeneous symmetric function $F$ of degree $n$ and a partition $\nu$ of $n$, let 
$\chi^F_\nu$ be the coefficient of $z_\nu^{-1} p_\nu$ in the expansion of $F$ in terms of the basis 
$\{z^{-1}_\nu p_\nu : \nu \in \pr( n)\}$ for the space of homogeneous symmetric functions of degree $n$.   That is, $ \chi^F_\nu$ is uniquely determined by
$$F = \sum_{\nu\in \pr(n) } \chi^F_\nu z^{-1}_\nu p_\nu.$$ 
Although we will not make use of this, we note that if $F$ is the Frobeneous characteristic of a class function of $S_n$ then $\chi^F_\nu$ is the value of the class function on permutations of cycle type $\nu$.  Recall that $\Omega$ denotes the stable principal specialization defined in Section~\ref{Eulerian}.

The following result is implicit in \cite{des}.
\begin{prop} \label{thdes} Suppose $f(q) \in \qq[q]$ and there exists  a homogeneous symmetric function $F$ of degree $n$ with coefficients in $\qq$ such that   
 $$f(q) =\prod_{i=1}^n (1-q^i)\,\, \Omega F.$$ Then for all $d,k \in \pp$ such that $n \in \{dk,dk+1\} $, 
 $$f(\omega_d) = \chi^F_\nu,$$
 where $\nu = d^{k}$ or   $\nu=1d^{k}$. 
\end{prop}

\begin{proof} By expanding $F$ in the power sum basis for the symmetric functions, we have, 
\begin{eqnarray} \nonumber f(q) &=& \prod_{i=1}^n (1-q^i) \sum_{\mu \in \pr(n)}  \chi^F_\mu z^{-1}_\mu\Omega \p_\mu \\ \label{desres}
 &=& \sum_{\mu \in \pr(n)} \chi^F_\mu  z^{-1}_\mu \,\, \frac{\prod_{i=1}^n (1-q^i)} {\prod_{i=1}^{l(\mu)} (1-q^{\mu_i} )}.
 \end{eqnarray}

It is shown in  \cite[Proposition 7.2]{des} that for all $\mu \in \pr(n)$,
$$T_\mu(q):= \frac{\prod_{i=1}^n (1-q^i)} {\prod_{i=1}^{l(\mu)} (1-q^{\mu_i} )}$$ is a polynomial in $q$ whose value at $\omega_d$ is given by
\begin{equation} \label{deseq} 
T_\mu(\omega_d) = \begin{cases}  z_\mu &\mbox{if } \mu = d^k  \mbox{ or } \mu = 1 d^k\\
 0 &\mbox{otherwise.} \end{cases} 
 \end{equation}
We  include a proof for the sake of completeness.  Since  $$T_\mu(q)=  \left [ \begin{array}{c} n \\ \mu_1,\dots ,\mu_{l(\mu)} \end{array} \right] _q \,\, \prod_{i=1}^{l(\mu)} \prod_{j=1}^{\mu_i-1}(1-q^j),$$ we see that $T_\mu(q)$ is a polynomial and  that if $T_\mu(\omega_d) \ne 0$ then $\mu_i \le d$ for all $i$.   Hence, in the case that $n=dk$,  it follows from 
 Proposition~\ref{multinom} that  $T_\mu(\omega_d) \ne 0$ only if $\mu_i = d$ for all $i$.  By Proposition~\ref{multinom}, 
 $$T_{d^k} (\omega_d) = k! (\prod_{j =1}^{d-1}(1-\omega_d^j))^k = k!d^k.$$  
 Similarly, in the case that $n=dk+1$, we use   (\ref{pascal})
and 
 Proposition~\ref{multinom} to show that that $T_\mu(\omega_d)$ equals $k!d^k$ if $\mu = 1 d^k$ and is $0$ otherwise.  Hence, in either case,
  (\ref{deseq}) holds.  Now by  plugging (\ref{deseq}) into (\ref{desres}) we obtain the desired result.
\end{proof}

We will use  Propostion~\ref{thdes} to  evaluate the cycle-type Eulerian numbers 
$a_{\lambda,j}(q)$ at all the $m^{th}$ roots of unity, where $m \in \{n-1,n\}$.    We see from Proposition~\ref{stabprop} that we already have the required symmetric function, namely $Q_{\lambda,j}$.  We thus obtain the  first   step in our proof of Theorem~\ref{main1}.
\begin{prop} \label{evalth}
Let $\lambda \in \pr( n)$ and let $d,k \in \pp$.  If $dk = n$ then
$$a_{\lambda,j}(\omega_d) = \chi^{Q_{\lambda,j}}_{d^k},$$
and if $dk =n-1$ then 
$$a_{\lambda,j}(\omega_d) = \chi^{Q_{\lambda,j}}_{1d^k}.$$
\end{prop}

In \cite{ShWa1} a formula for the coefficients $\chi^{Q_{(n),j}}_{\nu}$ is conjectured.   This formula turns out to be just what we need to prove Theorem~\ref{main1}.  In the next section we present the conjecture and its proof. 

\begin{remark} In \cite{ShWa1} it is conjectured that $Q_{\lambda,j}$ is the Frobenius characteristic of some representation of  $S_n$.  By Proposition~\ref{evalth} and Theorem~\ref{main1},  the restriction of the conjectured representation  to $G_n$ would necessarily  be isomorphic to the permutation representation for the action of  $G_n$  on $S_{\lambda,j}$.  \end{remark}

\section{The expansion of $Q_{(n),j} $} \label{psumsec}

In this section we present a key result of our paper (Theorem \ref{psum}), which was conjectured in \cite{ShWa1}.  For a power series $f(t)=\sum_{j \geq 0}a_jt^j$ and an integer $k$, let $f(t)_k$ be the power series obtained from $f(t)$ by erasing all terms $a_jt^j$ such that $\gcd(j,k) \neq 1$, so
\[
f(t)_k:=\sum_{\gcd(j,k)=1}a_jt^j.
\]
For example, if $f(t)=t+3t^2-5t^3+7t^4$ then $f(t)_2=t-5t^3$.

For a partition $\nu=(\nu_1,\ldots,\nu_k)$, set 
\[
g(\nu):=\gcd(\nu_1,\ldots,\nu_k).
\]

\begin{thm}[\cite{ShWa1}, Conjecture 6.5] \label{psum}
For $\nu=(\nu_1,\ldots,\nu_k) \in \pr(n)$, set
\[
G_\nu(t):=\left(tA_{k-1}(t)\prod_{i=1}^{k}[\nu_i]_t\right)_{g(\nu)}.
\]
Then
\begin{equation} \label{peq}
\sum_{j=0}^{n-1}Q_{(n),j}t^j=\sum_{\nu \in \pr(n)}z_\nu^{-1} G_\nu(t)\p_\nu.
\end{equation}
\end{thm}

Theorem \ref{psum} can be restated as follows.  Since $Q_{(n),j}$ is a homogeneous symmetric function of degree $n$, it 
can be expanded in the  basis $\{z^{-1}_\lambda p_\lambda : \lambda \in \pr(n)\}$.  Thus, the theorem says that the expansion coefficient of $z^{-1}_\nu p_\nu$ is $0$ if $\gcd(j,g(\nu)) \neq 1$, while if $\gcd(j,g(\nu))=1$ then the expansion coefficient  equals  the coefficient of $t^j$ in $tA_{l(\nu)-1}(t)\prod_{i=1}^{l(\nu)}[\nu_i]_t$. 

In order to prove Theorem~\ref{psum} we need two lemmas. As above, we write $\mu$ for the classical M\"obius function on $\pp$, and recall that
\begin{equation} \label{minv}
\sum_{d|n}\mu(d)=\left\{\begin{array}{ll} 1 & n=1, \\ 0 & \mbox{otherwise}. \end{array} \right.
\end{equation}

\begin{lemma} \label{resum}
For a partition $\nu=(\nu_1,\ldots,\nu_l)$, we have
\begin{equation} \label{resumeq}
G_\nu(t)=\sum_{d|g(\nu)}\mu(d)d^{l-1}t^dA_{l-1}(t^d)\prod_{i=1}^{l}\left[\frac{\nu_i}{d}\right]_{t^d}.
\end{equation}
\end{lemma}

\begin{proof}
It is known (and follows, for example, from \cite[Theorem 4.5.14]{Sta1}) that for any positive integer $k$ we have
\begin{equation} \label{euleq}
\frac{tA_{k-1}(t)}{(1-t)^k}=\sum_{j \geq 1}j^{k-1}t^j.
\end{equation}

It follows directly from the definition of $f(t)_d$ that for any power series $g,h$ and any $d \in \pp$ we have
\begin{equation} \label{powid}
\left(g(t)h(t^d)\right)_d=g(t)_dh(t^d).
\end{equation}

We see now that
\begin{eqnarray*}
G_\nu(t) & = & \left(tA_{l-1}(t)\prod_{i=1}^{l}\frac{1-t^{\nu_i}}{1-t}\right)_{g(\nu)} \\ & = & \left(\frac{tA_{l-1}(t)}{(1-t)^l}\prod_{i=1}^{l}(1-t^{\nu_i})\right)_{g(\nu)} \\ & = & \left(\sum_{j \geq 1}j^{l-1}t^j\right)_{g(\nu)}\prod_{i=1}^{l}(1-t^{\nu_i}) \\ & = & \sum_{j:\gcd(g(\nu),j)=1}j^{l-1}t^j\prod_{i=1}^{l}(1-t^{\nu_i}),
\end{eqnarray*}
the third equality above following from (\ref{euleq}) and (\ref{powid}).

Now
\begin{eqnarray*}
\sum_{d|g(\nu)}\mu(d)\sum_{a \geq 1}(ad)^{l-1}t^{ad} & = & \sum_{j \geq 1}j^{l-1}t^j\sum_{d|\gcd(j,g(\nu))}\mu(d) \\ & = & \sum_{j:\gcd(j,g(\nu))=1}j^{l-1}t^j,
\end{eqnarray*}
the second equality following from (\ref{minv}).  We see now that
\begin{eqnarray*}
G_\nu(t) & = & \left(\sum_{d|g(\nu)}\mu(d)\sum_{a \geq 1}(ad)^{l-1}t^{ad}\right)\prod_{i=1}^{l}(1-t^{\nu_i}) \\ & = & \left(\sum_{d|g(\nu)}\mu(d)d^{l-1}\frac{t^dA_{l-1}(t^d)}{(1-t^d)^l}\right)\prod_{i=1}^{l}(1-t^{\nu_i}) \\ & = & \sum_{d|g(\nu)}\mu(d)d^{l-1}t^dA_{l-1}(t^d)\prod_{i=1}^{l}\frac{1-t^{\nu_i}}{1-t^d},
\end{eqnarray*}
the second equality following from (\ref{euleq}).
\end{proof}

\begin{lemma} \label{exforlem}
We have
\begin{equation} \label{exfor}
\sum_{k \geq 0}\frac{A_k(t)}{k!}z^k=\exp\left(\sum_{l \geq 1}\frac{tA_{l-1}(t)}{l!}z^l\right).
\end{equation}
\end{lemma}

\begin{proof} We apply the exponential formula (see \cite[Corollary 5.1.6]{Sta2}) to the Eulerian polynomials. For any permutation $\sigma$  in $S_n$ let $\pi(\sigma)$ be the partition of the set $[n]$ whose blocks are the supports of the cycles in the cycle decomposition of $\sigma$.   Let $\Pi_n$ be the set of all partitions of the set $[n]$.  For any partition $\pi$ in $\Pi_n$ 
set
$$A_\pi(t) := \sum_{\scriptsize\begin{array}{c} \sigma \in S_n\\ \pi(\sigma) = \pi \end{array}} t^{\exc(\sigma)}.$$
Then $$A_n(t) = \sum_{\pi \in \Pi_n} A_{\pi}(t),$$
and $$A_\pi(t) = \prod_{i=1}^k A_{\{B_i\}}(t) =   \prod_{i=1}^k A_{(|B_i|)}(t),$$
where  $\pi = \{B_1,\dots,B_k\}$.
It therefore follows from the exponential formula that 
$$ \sum_{k\ge 0}  \frac{A_k(t)} {k! } z^k  = \exp (  \sum_{l\ge 1}  \frac{A_{(l)}(t)} {l! } z^l ).$$
To complete the proof we observe that 
\begin{equation} \label{circeul} A_{(l)}(t) = t A_{l-1}(t).\end{equation}
Indeed, for $\sigma \in S_{(l)}$, write $\sigma$ in cycle notation $(x_1,x_2,\ldots ,x_l)$ with $x_l = l$.   Now let $v(\sigma)=x_1\ldots x_{l-1}$, a permutation in $S_{l-1}$ in one line notation.  The excedance set of $\sigma$ is the union of $\{x_{l-1}\}$ and  $\{x_i : i \mbox{ is an ascent of } v(\sigma) \}$. 
Since $v$ is a bijection from $S_{(l)}$ to $S_{l-1}$, equation (\ref{circeul}) holds.
\end{proof}

\begin{proof}[Proof of Theorem~\ref{psum}]
We have
\begin{eqnarray} \nonumber
\sum_{n \geq 1}\sum_{j=0}^{n-1}Q_{(n),j}t^j & = & \LL\left[\sum_{i \geq 1}\sum_{j=0}^{i-1}Q_{i,j}t^j\right] \\ 
\nonumber & = & \LL\left[\sum_{k \geq 1}\sum_{\nu:l(\nu)=k}z_\nu^{-1}A_k(t)\prod_{h=1}^{k}[\nu_h]_t\p_{\nu_h}\right] \\ 
\nonumber & = & \sum_{d \geq 1}\frac{\mu(d)}{d}\sum_{i \geq 1}\frac{(-1)^{i-1}}{i}\p_d^i\left[\sum_{k \geq 1}\sum_{\nu:l(\nu)=k}z_\nu^{-1}A_k(t)\prod_{h=1}^{k}[\nu_h]_t\p_{\nu_h}\right] \\ 
\nonumber & = & \sum_{d \geq 1}\frac{\mu(d)}{d}\sum_{i \geq 1}\frac{(-1)^{i-1}}{i}\left(\sum_{k \geq 1}\sum_{\nu:l(\nu)=k}z_\nu^{-1}A_k(t^d)\prod_{h=1}^{k}[\nu_h]_{t^d}\p_{d\nu_h}\right)^i \\ 
\label{pfconj} & = & \sum_{d \geq 1} \frac{\mu(d)}{d}\log\left(1+\sum_{k \geq 1}\sum_{\nu:l(\nu)=k}z_\nu^{-1}A_k(t^d)\prod_{h=1}^{k}[\nu_h]_{t^d}\p_{d\nu_h}\right),
\end{eqnarray}
the first equality following from  (\ref{ple2}), the second from (\ref{qnj}), the third from the definition of $\LL$ and the fourth from (\ref{pple}).

For any $k \in \pp$, let ${\mathcal M}_k$ be the set of all sequences $a=(a_1,a_2,\ldots)$ of nonnegative integers such that $\sum_{i \geq 1}a_i=k$.  Then
\begin{eqnarray}\nonumber
\sum_{\nu:l(\nu)=k}z_\nu^{-1}\prod_{i=1}^{k}[\nu_i]_{t^d}\p_{d\nu_i} & = & \sum_{\nu:l(\nu)=k}\frac{1}{\prod_{r \geq 1}m_r(\nu)!}\prod_{r \geq 1}\left(\frac{[r]_{t^d}\p_{dr}}{r}\right)^{m_r(\nu)}
\\ \nonumber
 & = & \frac{1}{k!}\sum_{a \in {\mathcal M}_k}{{k} \choose {a_1,a_2,\ldots}}\prod_{r \geq 1}\left(\frac{[r]_{t^d}\p_{dr}}{r}\right)^{a_r} 
\\ \label{mceq}
 & = & \frac{1}{k!}\left(\sum_{r \geq 1}\frac{[r]_{t^d}\p_{dr}}{r}\right)^k.
\end{eqnarray}

We see now that
\begin{eqnarray*}
\sum_{n \geq 1}\sum_{j=0}^{n-1}Q_{(n),j}&=& \sum_{d \geq 1}\frac{\mu(d)}{d}\log\left(1+\sum_{k \geq 1}\frac{A_k(t^d)}{k!}\left(\sum_{r \geq 1}\frac{[r]_{t^d}\p_{dr}}{r}\right)^k\right)\\ &=&
 \sum_{d \geq 1} \frac{\mu(d)}{d}\sum_{k \geq 1} \frac{t^dA_{k-1}(t^d)}{k!}\left(\sum_{r \geq 1}\frac{[r]_{t^d}\p_{dr}}{r}\right)^k \\ & = & \sum_{d \geq 1}\frac{\mu(d)}{d}\sum_{k \geq 1}\sum_{\nu:l(\nu)=k}z_\nu^{-1}t^dA_{k-1}(t^d)\prod_{i=1}^{k}[\nu_i]_{t^d}\p_{d\nu_i} \\ & = & \sum_{k \geq 1}\sum_{d \geq 1}\mu(d)d^{k-1}t^dA_{k-1}(t^d)\sum_{\nu:l(\nu)=k}z_{d\nu}^{-1}\prod_{i=1}^{k}[\nu_i]_{t^d}\p_{d\nu_i} \\ & = & \sum_{k \geq 1}\sum_{\nu:l(\nu)=k}z_\nu^{-1}\p_\nu\sum_{d|g(\nu)}\mu(d)d^{k-1}t^dA_{k-1}(t^d)\prod_{i=1}^{k}\left[\frac{\nu_i}{d}\right]_{t^d} \\ & = & \sum_{k \geq 1}\sum_{\nu:l(\nu)=k}z_\nu^{-1}\p_\nu G_\nu(t).
\end{eqnarray*}
Indeed, first equality is obtained by combining  (\ref{pfconj}) and (\ref{mceq}), the second  equality is obtained from (\ref{exfor}), the third follows from (\ref{mceq}), the fourth and fifth are obtained by straightforward manipulations, and the last follows from (\ref{resumeq}).
\end{proof}

\section{The proof of Theorem \ref{main1}} \label{eqth1sec}

\subsection{The expansion coefficients $\chi^{Q_{\lambda,j}}_{d^k}$} 
To compute the expansion coefficients $\chi^{Q_{\lambda,j}}_{d^k}$, we will need to obtain results like Theorem \ref{psum} with the partition $(n)$ replaced by an arbitrary partition $\lambda$, but in such results we will only need the coefficients of power sum symmetric functions of the form $\p_{(d,\ldots,d)}$.  
We begin with a definition generalizing that of $f(t)_d$.  For a power series $f(t)=\sum_ja_jt^j$ and positive integers $b,c$, let $f(t)_{b,c}$ be the power series obtained from $f$ by erasing all terms $a_it^i$ such that $\gcd(i,b) \neq c$, so
\[
f(t)_{b,c}:=\sum_{\gcd(i,b)=c}a_it^i.
\]
For example, if $f(t)=1+2t+3t^2+4t^3+5t^4$ then $f(t)_{6,2}=3t^2+5t^4$.  If $c|b$ then for any power series $g,h$, we have 
\begin{equation} \label{ghbc}
\left(g(t)h(t^b)\right)_{b,c}=g(t)_{b,c}h(t^b).
\end{equation}
We will use the following result.

\begin{lemma} \label{ftbclem}
Let $k,b,c \in \pp$ and assume that $c|b$.  Then
\begin{equation} \label{ftbc}
\left(tA_{k-1}(t)[b]_t^k\right)_{b,c}=c^{k-1}\left(t^cA_{k-1}(t^c)[b/c]_{t^c}^k\right)_{b,c}.
\end{equation}
\end{lemma}

\begin{proof}
We have
\begin{eqnarray*}
\left(tA_{k-1}(t)[b]_t^k\right)_{b,c} & = & \left(\frac{tA_{k-1}(t)}{(1-t)^k}(1-t^b)^k\right)_{b,c} \\ & = & \left(\frac{tA_{k-1}(t)}{(1-t)^k}\right)_{b,c}(1-t^b)^k \\ & = & (1-t^b)^k\sum_{j:\gcd(j,b)=c}j^{k-1}t^j \\ & = & (1-t^b)^k\sum_{i:\gcd(i,b/c)=1}(ic)^{k-1}t^{ic} \\ & = & c^{k-1}(1-t^b)^k\sum_{i:\gcd(i,b/c)=1}i^{k-1}t^{ic} \\ & = & c^{k-1}(1-t^b)^k\left(\frac{t^cA_{k-1}(t^c)}{(1-t^c)^k}\right)_{b,c} \\ & = & c^{k-1}\left(t^cA_{k-1}(t^c)\frac{(1-t^b)^k}{(1-t^c)^k}\right)_{b,c} \\ & = & c^{k-1}\left(t^cA_{k-1}(t^c)[b/c]_{t^c}^{k}\right)_{b,c}.
\end{eqnarray*}
Indeed, the second and seventh equalities follow from (\ref{ghbc}), the third and sixth follow from (\ref{euleq}), and the rest are straightforward.
\end{proof}

We begin our computation of $\chi^{Q_{\lambda,j}}_{d^k}$ by  considering first the case where all parts of $\lambda$ have the same size. 
 For $\lambda, \nu \in \pr(n)$ and $j \in \{0,1,\dots,n-1\}$,  set
 $$ \chi_\nu^{\lambda,j} := \chi^{Q_{\lambda,j}}_\nu$$ and
\begin{equation} \label{pinu}
\pi_\nu :={{n} \choose {\nu_1,\ldots,\nu_{l(\nu)}}}\frac{1}{\prod_{j=1}^{n}m_j(\nu)!}.
\end{equation}

\begin{thm} \label{mrkdth}
Let $n,r,m,d,k \in \pp$ with $n=rm=dk$.  Set
\[
\pr(m;d,r):=\{\mu=(\mu_1,\ldots,\mu_{l(\mu)}) \in \pr(m):\mu_i|d|r\mu_i \mbox{ for all } i \in [l(\mu)]\}.
\]
Then
\begin{equation} \label{mrdk}
\sum_{j=0}^{n-1}\chi_{d^k}^{r^m,j}t^j=\sum_{\mu \in \pr(m;d,r)}\pi_{\frac{r}{d}\mu}\prod_{i=1}^{l(\mu)}\left(tA_{\frac{r}{d}\mu_i-1}(t)[d]_t^{\frac{r}{d}\mu_i}\right)_{d,\mu_i}.
\end{equation}
\end{thm}

\begin{proof}
Note that (\ref{ple}) implies that
\begin{equation} \label{reple}
\sum_{j=0}^{n-1}Q_{r^m,j}t^j=\h_m\left[\sum_{j=0}^{r-1}Q_{(r),j}t^j\right].
\end{equation}
Now
\begin{eqnarray*}
\h_m\left[\sum_{j=0}^{r-1}Q_{(r),j}t^j\right] & = & \sum_{\mu \in \pr(m)}z_\mu^{-1}\p_\mu\left[\sum_{j=0}^{r-1}Q_{(r),j}t^j\right] \\ & = & \sum_{\mu \in \pr(m)}z_\mu^{-1}\prod_{i=1}^{l(\mu)}\p_{\mu_i}\left[\sum_{\nu \in \pr(r)}z_\nu^{-1}G_\nu(t)\p_\nu\right] \\ & = & \sum_{\mu \in \pr(m)}z_\mu^{-1}\prod_{i=1}^{l(\mu)}\sum_{\nu \in \pr(r)}z_\nu^{-1}G_\nu(t^{\mu_i})\p_{\mu_i\nu}.
\end{eqnarray*}
Indeed, the first equality follows from  the well known expansion (see any of \cite{Mac,Sag,Sta2}) of $\h_m$ in the power sum basis, the second from Theorem \ref{psum} and the third from (\ref{pple}).

On the other hand, it follows from the definition of $\chi^{\lambda,j}_\nu$ that
\[
\sum_{j=0}^{n-1}Q_{r^m,j}t^j=\sum_{\nu \in \pr(n)}z_\nu^{-1}\sum_{j \geq 0}\chi^{r^m,j}_\nu\p_\nu t^j.
\]
We see now equating the coefficients of $\p_{d^k}$ on both sides of (\ref{reple}) yields
\begin{equation} \label{pdk}
\sum_{j=0}^{n-1}\chi^{r^m,j}_{d^k}t^j=z_{d^k}\sum_{\mu \in \pr(m;d,r)}z_\mu^{-1}\prod_{i=1}^{l(\mu)}z^{-1}_{\left(\frac{d}{\mu_i}\right)^{r\mu_i/d}}G_{\left(\frac{d}{\mu_i}\right)^{r\mu_i/d}}(t^{\mu_i}).
\end{equation}

Now for all $\mu \in \pr(m;d,r)$, we have
\begin{eqnarray} \nonumber
z_{d^{k}}z_\mu^{-1}\prod_{i=1}^{l(\mu)}z^{-1}_{\left(\frac{d}{\mu_i}\right)^{r\mu_i/d}} &=& \frac{d^kk!}{\prod_{j=1}^mm_j(\mu)!\prod_{i=1}^{l(\mu)}\mu_i\left(\frac{r\mu_i}{d}\right)!\left(\frac{d}{\mu_i}\right)^{r\mu_i/d}} \\  \nonumber & &
\\ \nonumber &=&
 = \frac{k!\prod_{i=1}^{l(\mu)}\mu_i^{(r\mu_i/d)-1}}{\prod_{j=1}^{m}m_j(\mu)!\prod_{i=1}^{l(\mu)}\left(\frac{r\mu_i}{d}\right)!}
\\ \nonumber & &  \\
 \label{zpi}
&=& \pi_{\frac{r}{d}\mu}\prod_{i=1}^{l(\mu)}\mu_i^{(r\mu_i/d)-1},
\end{eqnarray}

and
\begin{eqnarray} \nonumber
G_{\left(\frac{d}{\mu_i}\right)^{r\mu_i/d}}(t^{\mu_i}) &=&\left(tA_{\left(r\mu_i/d\right)-1}(t)\left[d/\mu_i\right]_t^{r\mu_i/d}\right)_{d/\mu_i} |_{t=t^{\mu_i}}
\\
 \label{gdmu}
&=& \left(t^{\mu_i}A_{\left(r\mu_i/d\right)-1}(t^{\mu_i})\left[d/\mu_i\right]_{t^{\mu_i}}^{r\mu_i/d}\right)_{d,\mu_i}.
\end{eqnarray}
We now have
\begin{eqnarray*}
\sum_{j=0}^{n-1}\chi_{d^k}^{r^m,j}  t^j & = & \sum_{\mu \in \pr(m;d,r)}\pi_{\frac{r}{d}\mu}\prod_{i=1}^{l(\mu)}\mu_i^{(r\mu_i/d)-1}\left(t^{\mu_i}A_{\left(r\mu_i/d\right)-1}(t^{\mu_i})\left[d/\mu_i\right]_{t^{\mu_i}}^{r\mu_i/d}\right)_{d,\mu_i} \\ & = & \sum_{\mu \in \pr(m;d,r)}\pi_{\frac{r}{d}\mu}\prod_{i=1}^{l(\mu)}\left(tA_{\frac{r}{d}\mu_i-1}(t)[d]_t^{\frac{r}{d}\mu_i}\right)_{d,\mu_i},
\end{eqnarray*}
the first equality being obtained by substituting (\ref{zpi}) and (\ref{gdmu}) into (\ref{pdk}), and the second following from Lemma \ref{ftbclem}.
\end{proof}

We use Theorem \ref{mrkdth} to handle general $\lambda$.

\begin{thm} \label{chilamnu}
Say $\lambda \in \pr(n)$ and $n=kd$.
\begin{enumerate}
\item If there is some $r \in [n]$ such that $d$ does not divide $rm_r(\lambda)$ then 
$\chi^{\lambda,j}_{d^k}=0$ for all $0 \leq j \leq n-1$.
\item If $d$ divides $rm_r(\lambda)$ for all $r \in [n]$ then
\begin{equation*}
\sum_{j=0}^{n-1}\chi_{d^k}^{\lambda,j} t^j={{\frac{n}{d}} \choose {\frac{1m_1(\lambda)}{d},\frac{2m_2(\lambda)}{d},\ldots,\frac{nm_n(\lambda)}{d}}} \prod_{r=1}^{n}\sum_{j=0}^{rm_r(\lambda)-1}\chi_{d^{rm_r(\lambda)/d}}^{r^{m_r(\lambda)},j}\,\, t^j.
\end{equation*}
\end{enumerate} 
\end{thm}

\begin{proof}
It follows directly from (\ref{ple}) that
\begin{equation} \label{larmr}
\sum_{j=0}^{n-1}Q_{\lambda,j}t^j=\prod_{r=1}^{n}\sum_{j=0}^{rm_r(\lambda)-1}Q_{r^{m_r(\lambda)},j}\,\,t^j.
\end{equation}
Expressing both sides of  (\ref{larmr}) in terms of the power sum basis, we get
\begin{equation} \label{frlarmr}
\sum_{\mu \in \pr(n)}z_\mu^{-1}\sum_{j=0}^{n-1}\chi_\mu^{\lambda,j} t^j\p_\mu=\prod_{r=1}^{n}\sum_{\nu \in \pr(rm_r(\lambda))}z_\nu^{-1}\sum_{j=0}^{rm_r(\lambda)-1}\chi_\nu^{r^{m_r(\lambda)},j}\,\,t^j\p_\nu.
\end{equation}
Equating coefficients of $\p_{d^k}$ in (\ref{frlarmr}) we see that if $d$ does not divide every $rm_r(\lambda)$ then $\chi_{d^k}^{\lambda,j}=0$ for all $j$, while if $d$ divides every $rm_r(\lambda)$ then
\begin{eqnarray*}
\sum_{j=0}^{n-1}\chi_{d^k}^{\lambda,j} t^j & = & z_{d^k}\prod_{r=1}^{n}z_{d^{rm_r(\lambda)/d}}^{-1}\sum_{j=0}^{rm_r(\lambda)-1}\chi_{d^{rm_r(\lambda)/d}}^{r^{m_r(\lambda)},j}\,\,t^j 
\\ & = & \frac{d^{n/d}(n/d)!}{\prod_{r=1}^{n}d^{rm_r(\lambda)/d}(rm_r(\lambda)/d)!}\prod_{r=1}^{n}\sum_{j=0}^{rm_r(\lambda)-1}\chi_{d^{rm_r(\lambda)/d}}^{r^{m_r(\lambda)},j}\,\,t^j \\ & = & {{\frac{n}{d}} \choose {\frac{1m_1(\lambda)}{d},\frac{2m_2(\lambda)}{d},\ldots,\frac{nm_n(\lambda)}{d}}} \prod_{r=1}^{n}\sum_{j=0}^{rm_r(\lambda)-1}\chi_{d^{rm_r(\lambda)/d}}^{r^{m_r(\lambda)},j}\,\, t^j. 
\end{eqnarray*}
\end{proof}

\subsection{The permutation character $\theta^{\lambda,j}$ of $G_n$}

Note that, upon considering cycle notation for elements of $S_n$, it is straightforward to show that if $\sigma \in \cjl$ then $\gamma_n^{-1}\sigma\gamma_n \in \cjl$.  Thus the claim in Theorem \ref{main1} that $G_n$ acts on $\cjl$ is correct.   Let $\theta^{\lambda,j}$ denote     the permutation character of the action of $G_n$ on   $\cjl$.  Hence,  $\theta^{\lambda,j}(\tau) $ is the number of elements of $\cjl$ centralized by $\tau \in G_n$.    For $\nu \in \pr(n)$, let  $\theta^{\lambda,j}_\nu= \theta^{\lambda,j}(\tau)$, 
where $\tau$ is any permutation of cycle type $\nu$.    Since all $\tau \in G_n$ have cycle type  of the form $d^k$,  where $dk=n$,  we need only concern ourselves with $\nu=d^k$.

With Theorems \ref{mrkdth} and \ref{chilamnu} in hand, we now produce matching results for the permutation characters $\theta^{\lambda,j}$ of $G_n$.  Again we begin with the case where $\lambda=r^m$ for some divisor $r$ of $n$.  Before doing so, we derive, in the form most useful for our arguments, some known facts about centralizers in $S_n$ of elements of $G_n$, along with straightforward consequences of these facts.

Fix positive integers $n,k,d$ with $n=kd$.  Set $\tau=\gamma_n^{-k} \in G_n$.  Note that $C_{S_n}(\tau)=C_{S_n}(\gamma_n^k)$.  For $i \in [n]$, we have
\[
i\tau=\left\{ \begin{array}{ll} i-k & i>k, \\ i-k+n & i \leq k. \end{array} \right.
\]
Now $\tau$ has cycle type $d^k$, and we can write $\tau$ as the product of $k$ $d$-cycles, $\tau=\tau_1\ldots\tau_k$, where $\tau_i$ has support
\[ 
X_i:=\{j \in [n]:j \equiv i \bmod k\}.\]

It follows that if $\sigma \in C_{S_n}(\tau)$ then for each $i \in [k]$ there is some $j \in [k]$ such that $X_i\sigma=X_j$.  Thus we have an action of $C_{S_n}(\tau)$ on $\{X_1,\ldots,X_k\}$, which gives rise to a homomorphism $$\Phi:C_{S_n}(\tau) \rightarrow S_k.$$  Given $\rho \in S_k$, define $\hat{\rho} \in S_n$ to be the element that, for $r \in [k]$ and $q \in \{0,\ldots,d-1\}$, maps $r+qk$ to $r\rho+qk$.  It is straightforward to check that $\hat{\rho} \in C_{S_n}(\tau)$ and $\Phi(\hat{\rho})=\rho$.  Moreover, if we set
\[
R:=\{\hat{\rho}:\rho \in S_k\},
\]
then $R \leq C_{S_n}(\tau)$ and the restriction of $\Phi$ to $R$ is an isomorphism.  It follows that if we set $K={\rm kernel}(\Phi)$ then $C_{S_n}(\tau)$ is the semidirect product of $K$ and $R$.  Now
\[
K=\{\sigma \in C_{S_n}(\tau):X_i\sigma=X_i \mbox{ for all } i \in [k]\}=\prod_{i=1}^{k}C_{S_{X_i}}(\tau_i).
\]
Since every $d$-cycle in $S_d$ generates its own centralizer in $S_d$, we have
\[
K=\prod_{i=1}^{k}\langle \tau_i \rangle=\left\{\prod_{i=1}^{k}\tau_i^{e_i}:e_1,\ldots,e_k \in \{0,\ldots,d-1\}\right\}.
\]
Now, given $\rho \in S_k$ and $e_1,\ldots,e_k \in \{0,\ldots,d-1\}$, set
\[
\sigma:=\tau_1^{e_1}\ldots\tau_k^{e_k}\hat{\rho} \in C_{S_n}(\tau).
\]
For $r \in [k]$ and $q \in \{0,\ldots,d-1\}$, we have (with $\sigma$ acting on the right)
\begin{equation} \label{centact}
(r+qk)\sigma=\left\{ \begin{array}{ll} r\rho+(q-e_r)k & q \geq e_r, \\ r\rho+(q-e_r)k+n & q<e_r. \end{array} \right.
\end{equation}
It follows that $r+qk \in \Exc(\sigma)$ if and only if either $q < e_r$ or $e_r=0$ and $r<r\rho$.  We collect in the next lemma some useful consequences of what we have just seen.

\begin{lemma} \label{centlem}
Let $n=dk$ and let $\tau=\gamma_n^{-k}$.  Let $\sigma \in C_{S_n}(\tau)$.  Then there exist unique $\rho \in S_k$ and $e_1,\ldots,e_k \in \{0,\ldots,d-1\}$ such that
\[
\sigma=\tau_1^{e_1}\ldots\tau_k^{e_k}\hat{\rho},
\]
and if we define $E_0$ to be the number of $r \in [k]$ such that $e_r=0$ and $r \in \Exc(\rho)$, then
\begin{equation} \label{centeq}
\exc(\sigma)=dE_0+\sum_{i=1}^{k}e_i.
\end{equation} 
\end{lemma}

Note that  the unique $\rho \in S_k$ of Lemma~\ref{centlem} is  equal to $\Phi(\sigma)$ defined above.
For $\mu \in \pr(k)$ and any divisor $r$ of $n$, set
\[
\rc_\mu:=\{\sigma \in C_{S_n}(\tau):\Phi(\sigma) \in \cm\}
\]
and
\[
\rc_{\mu,r}:=\rc_\mu \cap S_{r^{n/r}},
\]
so $\rc_{\mu,r}$ consists of those $\sigma \in C_{S_n}(\tau)$ such that $\sigma$ has cycle type $r^{n/r}$ and $\Phi(\sigma)$ has cycle type $\mu$.  

\begin{lemma} \label{ckr}
For any divisor $r$ of $n$, we have
\[
\sum_{\sigma \in \rc_{(k),r}}t^{\exc(\sigma)}=\left\{ \begin{array}{ll} \left(tA_{k-1}(t)[d]_t^k\right)_{d,\frac{n}{r}} & \mbox{if } k|r, \\ 0 & \mbox{otherwise} \end{array} \right.
\]
\end{lemma}

\begin{proof}
We begin by showing that
\begin{equation} \label{disu}
\rc_{(k)}=\biguplus_{k|r|n}\rc_{(k),r}.
\end{equation}

Certainly the union on the right side of (\ref{disu}) is contained in $\rc_{(k)}$, so we prove that this union contains $\rc_{(k)}$.  Let $\sigma \in \rc_{(k)}$.  By Lemma \ref{centlem}, we have
\[
\sigma=\tau_1^{e_1}\ldots\tau_k^{e_k}\hat{\rho}
\]
for unique $\rho \in S_{(k)}$ and $e_1,\ldots,e_k \in \{0,\ldots,d-1\}$.   It follows from (\ref{centact}) that for each $j \in [n]$ we have
\begin{equation} \label{sk}
j\sigma^k \equiv j-k\sum_{i=1}^{k}e_i \bmod n.
\end{equation}
Moreover, if $j\sigma^l \equiv j \bmod k$ then $k|l$.   Hence each cycle length in the cycle decomposition of $\sigma$ is a multiple of $k$.

We claim that all cycles in the cycle decomposition of $\sigma$ have length $sk$, where $s$ is the order of  $k\sum_{i=1}^{k}e_i$ in $\zz_n$.    Indeed, it follows from (\ref{sk}) that  for all $j \in [n]$,
$$ j \sigma^{sk} \equiv j -sk\sum_{i=1}^k e_i \equiv j \bmod n ,$$  which implies that $j\sigma^{sk} = j$.  Hence the order of $\sigma$ in $S_n$ divides $sk$. It follows that every cycle length in the cycle decomposition of $\sigma$ divides $sk$.   Now we need only show that $sk$ divides the length of each cycle. Suppose $\alpha$ is a cycle of  length $r$ and $j$ is an element in the support of $\alpha$.  We have $k|r$ since $k$ divides the length of every cycle.
Again using 
(\ref{sk}) we have, 
$$j = j\sigma^r = j ( \sigma^k)^{r/k}  \equiv j - \frac r k k \sum_{i=1}^k e_i \mod n ,$$ which implies that $( r/ k) k \sum_{i=1}^k e_i \equiv 0 \bmod n$.  Thus $s$, the order of $k \sum_{i=1}^k e_i $, divides $r/k$, which implies that $sk$ divides $r$.  We have therefore shown that $sk$ divides the length of every cycle, and since we have already shown that every cycle length divides $sk$, we conclude that all cycles in the cycle decomposition of $\sigma$ have  the same length $sk$, that is, $\sigma \in \rc_{(k),r}$ for some $r$ satisfying $k|r|n$, as claimed in (\ref{disu}).  

We have also shown  that  $\rc_{(k), r} = \emptyset$ if $k $ does not divide $r$.  Thus the claim of the lemma holds when $k$ does not divide $r$.

Next we show that if $\sigma \in \rc_{(k),r}$ then
\begin{equation} \label{gcdckr}
\gcd(\exc(\sigma),d)=\frac{n}{r}.
\end{equation}
As above, write $\sigma=\tau_1^{e_1}\ldots\tau_k^{e_k}\hat{\rho}$.  Since $d|n$, it follows from (\ref{centeq}) that
\begin{equation} \label{gcdgcd}
\gcd(\exc(\sigma),d)=\gcd\left(\sum_{i=1}^{k}e_i,d\right).
\end{equation}
Since $k\sum_{i=1}^{k}e_i$ has order $r/k$ in $(\zz_n,+)$, we have that
\[
\frac{r}{k}=\frac{n}{\gcd\left(n,k\sum_{i=1}^{k}e_i\right)}=\frac{d}{\gcd\left(d,\sum_{i=1}^{k}e_i\right)},
\]
the first equality following from simple facts about modular arithmetic and the second from the fact that $n=dk$.  Now we have
\begin{equation} \label{gcdd}
\gcd\left(d,\sum_{i=1}^{k}e_i\right)=\frac{kd}{r},
\end{equation}
and combining (\ref{gcdd}) with (\ref{gcdgcd}) gives (\ref{gcdckr}).

Combining (\ref{disu}) and (\ref{gcdckr}), we get
\begin{equation} \label{dnr}
\sum_{\sigma \in \rc_{(k),r}}t^{\exc(\sigma)}=\left(\sum_{\sigma \in \rc_{(k)}}t^{\exc(\sigma)}\right)_{d,\frac{n}{r}}
\end{equation}
for each divisor $r$ of $n$.

For $\rho \in S_k$ and $i \in [k]$, set
\[
f_{\rho,i}(t):=\left\{ \begin{array}{ll} t[d]_t & \mbox{if } i \in \Exc(\rho), \\ {[d]_t} & \mbox{otherwise} \end{array} \right.
\]
Then
\[
\sum_{\sigma \in \Phi^{-1}(\rho)}t^{\exc(\sigma)}=\prod_{i=1}^{k}f_{\rho,i}(t)=t^{\exc(\rho)}[d]_t^k,
\]
the first equality following from Lemma \ref{centlem}.  It follows now from (\ref{circeul}) that
\begin{equation} \label{akexc}
\sum_{\sigma \in \rc_{(k)}}t^{\exc(\sigma)}=tA_{k-1}(t)[d]_t^k,
\end{equation}
and combining (\ref{dnr}) and (\ref{akexc}) yields the lemma.
\end{proof}

\begin{lemma} \label{excprod}
For divisors $k,r$ of $n$ and $\mu \in \pr(k)$, we have
\[
\sum_{\sigma \in \rc_{\mu,r}}t^{\exc(\sigma)}=\pi_\mu\prod_{i=1}^{l(\mu)}\sum_{\sigma \in \rc_{(\mu_i),r}}t^{\exc(\sigma)},
\]
where   $\pi_\mu$ is  defined as in (\ref{pinu}).
\end{lemma}

\begin{proof}
Given $\sigma \in \rc_{\mu,r}$, we write as usual $\sigma=\tau_1^{e_1}\ldots\tau_k^{e_k}\hat{\rho}$.  Now $\rho \in S_k$ has cycle type $\mu$, so we can write $\rho=\rho_1\ldots\rho_{l(\mu)}$ as a product of disjoint cycles whose lengths form the partition $\mu$.  For $i \in [l(\mu)]$, let $B_i$ be the support of $\rho_i$.  We may assume that $|B_i|=\mu_i$ for all $i$.  Set
\[
\beta(\sigma):=\{B_1,\ldots,B_{l(\mu)}\},
\]
so $\beta(\sigma)$ is a partition of $[k]$.  For $i \in l(\mu)$, set
\[
\sigma_i:=\left(\prod_{j \in B_i}\tau_j^{e_j}\right)\widehat{\rho_i} \in S_n.
\]
The supports of both $\widehat{\rho_i}$ and $\prod_{j \in B_i}\tau_j^{e_j}$ are contained in
\[
\{j+qk:j \in B_i,0 \leq q \leq \frac{n}{k}-1\}.
\]
It follows that $\widehat{\rho_i}$ and $\Pi_{j \in B_h} \tau_j^{e_j}$ commute for all $i \ne h$,
so
\[
\sigma=\prod_{i=1}^{l(\mu)}\sigma_i.
\]
Moreover,
\begin{equation} \label{ssi}
\exc(\sigma)=\sum_{i=1}^{l(\mu)}\exc(\sigma_i).
\end{equation}
For $i \in [l(\mu)]$, define $f_i$ to be the unique order preserving bijection from $B_i$ to $[\mu_i]$, and set
\[
\overline{\sigma}_i:=f_i^{-1}\sigma_if_i.
\]
Then, for each $i \in [l(\mu)]$, we have $\overline{\sigma}_i \in \rc_{(\mu_i),r}$ and 
\begin{equation} \label{oss}
\exc(\overline{\sigma}_i)=\exc(\sigma_i).  
\end{equation}
Let $\Pi_\mu$ be the set of partitions of $[k]$ that have $m_j(\mu)$ blocks of size $j$ for each $j$. 
For each partition $X\in \Pi_\mu$, set
\[
\rc_X:=\{\sigma \in \rc_{\mu,r}:\beta(\sigma)=X\}.
\]
The map from $\rc_X$ to $\prod_{i=1}^{l(\mu)}\rc_{(\mu_i),r}$ sending $\sigma$ to $(\overline{\sigma}_1,\ldots,\overline{\sigma}_{l(\mu)})$ is a bijection.  Given (\ref{ssi}) and (\ref{oss}), we see that
\begin{eqnarray*}
\sum_{\sigma \in \rc_{\mu,r} } t^{\exc(\sigma} &=& \sum_{X \in \Pi_\mu} \sum_{\sigma \in \rc_X}t^{\exc(\sigma)}\\ &=& |\Pi_\mu| \prod_{i=1}^{l(\mu)}\sum_{\rho \in \rc_{(\mu_i),r}}t^{\exc(\rho)}.  
\end{eqnarray*}
It is straightforward to see that $|\Pi_\mu| = \pi_\mu$,
so the lemma follows.
\end{proof}

\begin{thm} \label{mrkdth2}
Let $n,r,m,d,k \in \pp$ with $n=rm=dk$.  As in Theorem \ref{mrkdth}, let
\[
\pr(m;d,r)=\{\mu=(\mu_1,\ldots,\mu_{l(\mu)}) \in \pr(m):\mu_i|d|r\mu_i \mbox{ for all } i \in [l(\mu)]\}.
\]
Then
\begin{equation} \label{mrdk2}
\sum_{j=0}^{n-1}\theta_{d^k}^{r^m,j}t^j=\sum_{\mu \in \pr(m;d,r)}\pi_{\frac{r}{d}\mu}\prod_{i=1}^{l(\mu)}\left(tA_{\frac{r}{d}\mu_i-1}(t)[d]_t^{\frac{r}{d}\mu_i}\right)_{d,\mu_i}.
\end{equation}
\end{thm}

\begin{proof}
We have
\begin{eqnarray*}
\sum_{j=0}^{n-1}\theta_{d^k}^{r^m,j}t^j & = & \sum_{\sigma \in C_{S_n}(\gamma_n^{k}) \cap S_{r^m}}t^{\exc(\sigma)} \\ & = & \sum_{\mu \in \pr(k)}\sum_{\sigma \in \rc_{\mu,r}}t^{\exc(\sigma)} \\ & = & \sum_{\mu \in \pr(k)}\pi_\mu\prod_{i=1}^{l(\mu)}\sum_{\sigma \in \rc_{(\mu_i),r}}t^{\exc(\sigma)} \\ & = & \sum_{\mu \in \pr(k;r,d)}\pi_\mu\prod_{i=1}^{l(\mu)}\left(tA_{\mu_i-1}(t)[d]_t^{\mu_i}\right)_{d,d\mu_i/r}.
\end{eqnarray*}
Indeed, the first two equalities follow immediately from the definitions of $\theta^{{r^m},j}$ and $\rc_{\mu,r}$, respectively, while the third follows from Lemma \ref{excprod} and the fourth from Lemma \ref{ckr}. 

Now for $\mu \in \pr(k;r,d)$, set $\nu:=\nu(\mu):=\frac{d}{r}\mu$, so $\mu=\frac{r}{d}\nu$.  Now $\frac{d}{r}k=m$ and, since $\mu_i|r|d\mu_i$, we have $\nu_i|d|r\nu_i$ for all $i$.  Thus $\nu \in \pr(m;d,r)$.  From this we see that the map $\mu \mapsto \nu(\mu)$ is a bijection from $ \pr(k;r,d)$ to $ \pr(m;d,r)$. Thus we have
\[
\sum_{j=0}^{n-1}\theta_{d^k}^{r^m,j}t^j=\sum_{\nu \in \pr(m;d,r)}\pi_{\frac{r}{d}\nu}\prod_{i=1}^{l(\nu)}\left(tA_{\frac{r}{d}\nu_i-1}(t)[d]_t^{\frac{r}{d}\nu_i}\right)_{d,\nu_i}
\]
as claimed.
\end{proof}

\begin{thm} \label{last}
Say $\lambda \in \pr(n)$ and $n=kd$.  
\begin{enumerate}
\item If there is some $r \in [n]$ such that $d$ does not divide $rm_r(\lambda)$ then $\theta^{\lambda,j}_{d^k}=0$ for all $0 \leq j \leq n-1$.
\item If $d$ divides $rm_r(\lambda)$ for all $r \in [n]$ then
\begin{equation*}
\sum_{j=0}^{n-1}\theta_{d^k}^{\lambda,j}t^j={{\frac{n}{d}} \choose {\frac{1m_1(\lambda)}{d},\frac{2m_2(\lambda)}{d},\ldots,\frac{nm_n(\lambda)}{d}}} \prod_{r=1}^{n}\sum_{j=0}^{rm_r(\lambda)-1}\theta_{d^{rm_r(\lambda)/d}}^{r^{m_r(\lambda)},j} t^j.
\end{equation*}
\end{enumerate} 
\end{thm}

\begin{proof}
Let $\sigma \in S_\lambda$, so $\sigma$ can be written as a product of disjoint cycles in which there appear exactly $m_r(\lambda)$ $r$-cycles for each $r \in [n]$.  For each such $r$, let $\sigma_r$ be the product of all these $m_r(\lambda)$ $r$-cycles, and let $B_r$ be the support of $\sigma_r$.  
If  $\sigma \in C_{S_n}(\gamma_n^k)$ then $\gamma_n^k$ commutes with each $\sigma_r$.  It follows that $\gamma_n^k$ stabilizes each $B_r$ setwise.  Therefore, for each $r \in [n]$, there is some $Y_r \subseteq [k]$ such that 
$$B_r=\biguplus_{i\in Y_r} X_i,$$  where 
$X_i = \{h \in [n] : h\equiv i \mod k\}$.   Since $|X_i| = d$ for all $i$,   it follows that  $rm_r(\lambda)=|B_r|=d |Y_r|$, so (1) holds.

For each $r \in [n]$, let $\beta_r \in S_{B_r}$ act as $\gamma_n^k$ does on $B_r$.  We have $\gamma_n^k=\prod_{r=1}^n\beta_r$, and $\beta_r$ commutes with $\sigma_r$ for all $r$.  For each $r \in [n]$, let $f_r$ be the unique order preserving bijection from $B_r$ to $[rm_r(\lambda)]$.  Direct calculation shows that for each $r$, we have
\begin{equation} \label{ftf}
f_r^{-1}\beta_rf_r=\gamma_{rm_r(\lambda)}^{rm_r(\lambda)/d}.
\end{equation}
Also,
\begin{equation} \label{excequiv}
\exc(\sigma)=\sum_{r=1}^{n}\exc(\sigma_r)=\sum_{r=1}^{n}\exc(f_r^{-1}\sigma_rf_r).
\end{equation} 

On the other hand suppose we are given an ordered $n$-tuple $(Y_1,\ldots,Y_n)$ of subsets of $[k]$ such that 
\begin{itemize}
\item[(a)] $|Y_r|=rm_r(\lambda)/d$ for each $r \in [n]$, and
\item[(b)] $[k]=\biguplus_{r=1}^{n}Y_r$,
\end{itemize}
and we set $B_r=\uplus_{i\in Y_r} X_i$ for each $r$. 
Then each $B_r$ is $\gamma_n^k$-invariant, and if we set  $\beta_r$ equal to the restriction of $\gamma_n^k$ to $B_r$, we can obtain $\sigma \in C_{S_n}(\gamma_n^k) \cap S_\lambda$ by choosing, for each $r$, any
$\sigma_r \in S_{B_r}$ of shape $r^{m_r(\lambda)}$ commuting with $\beta_r$ and setting $\sigma=\prod_{r=1}^{n}\sigma_r$. 
The number of $n$-tuples satisfying (a) and (b) is 
\[
{{\frac{n}{d}} \choose {\frac{1m_1(\lambda)}{d},\frac{2m_2(\lambda)}{d},\ldots,\frac{nm_n(\lambda)}{d}}},
\]
and the theorem now follows from (\ref{ftf}) and (\ref{excequiv}).
\end{proof}

Comparing Theorem \ref{mrkdth2} with Theorem \ref{mrkdth} and then comparing Theorem \ref{last} with Theorem \ref{chilamnu}, we obtain 
$$\chi_{d^k}^{Q_{\lambda,j}} = \theta_{d^k}^{\lambda,j}$$
for all $\lambda \in \pr( n)$, $j\in \{0,1,\dots,n-1\}$, and $d,k \in \pp$ such that $dk=n$.
Theorem~\ref{main1} now follows from Proposition~\ref{evalth}.

\section{Some additional results} \label{finsec}

As mentioned in the Introduction, Theorem~\ref{main1} is a refinement of  (\ref{majexcfixeq}). In this section we show that  the less refined result  can also be obtained as a consequence of  \cite[Corollary 4.3]{ShWa1}, which states that

\begin{equation} \label{formAn} A_n^{\maj,\exc, \fix}(q,t,s)  =   \sum_{m = 0}^{\lfloor {n \over 2} \rfloor}  \!\!\!\!\sum_{\scriptsize
\begin{array}{c} k_0\ge 0  \\ k_1,\dots, k_m \ge 2 \\ \sum k_i = n
\end{array}} \left[\begin{array}{c} n \\k_0,\dots,k_m\end{array}\right]_q\,\,
s^{k_0}
\prod_{i=1}^m tq[k_i-1]_{tq}.\end{equation}
  Although the alternative proof does not directly involve the Eulerian quasisymmetric functions, the proof of (\ref{formAn})  given in  \cite{ShWa1} does.  Hence the Eulerian quasisymmetric functions play an indirect role.   In this section we also prove the identity (\ref{otherform}) mentioned in the introduction and as a consequence obtain another cyclic sieving result.  

\begin{thm}\label{3pol} Let $dk=n$. Then
the following expressions  are all equal.
\begin{enumerate}
\item[(i)] $A_n^{\maj,\exc,\fix}(\omega_d,t \omega_d^{-1},s)$\\

\item[(ii)] $\sum_{\sigma \in C_{S_n}(\gamma_n^k)}t^{\exc(\sigma)}s^{\fix(\sigma)}$\\

\item[(iii)] $A^{\exc,\fix}_k(t, \frac{s^d+t[d-1]_t}{[d]_t})[d]_t^k$.
\end{enumerate}
\end{thm}

\begin{proof}  ((ii)=(iii)) 
For $\rho \in S_k$ and $i \in [k]$ set
$$f_{\rho,i}(t,s) := \begin{cases} t[d]_t &\mbox{if }  i \in \Exc(\rho) \\ s^d + t[d-1]_t  &\mbox{if } i \in \fx(\rho) \\ [d]_t &\mbox{otherwise.} \end{cases}$$
It follows from Lemma~\ref{centlem} that
\begin{eqnarray*}
\sum_{\sigma \in \Phi^{-1}(\rho)}t^{\exc(\sigma)}s^{\fix(\sigma)}&=& \prod_{i=1}^{k}f_{\rho,i}(t,s) \\ &= & t^{\exc(\rho)}[d]_t^{k-\fix(\rho)}(s^d+t[d-1]_t)^{\fix(\rho)}.
\end{eqnarray*}
By summing over all $\rho \in S_k$, we obtain the equality of the expressions in (ii) and (iii).

((i) = (iii))  By setting $q=1$ in (\ref{expgeneqfix}) we obtain 
\begin{eqnarray} \nonumber \sum_{k\ge 0} A_k^{\exc,\fix}(t,\frac{s^d+t[d-1]_t}{[d]_t}) [d]_t^k \frac{ z^k}{k!} & =&  \frac{(1-t) e^{(s^d+t[d-1]_t)z}} {e^{t[d]_t z} - t e^{[d]_t z}}
\\ &=& \nonumber
\frac{(1-t)e^{s^dz}}{e^{(t[d]_t -t[d-1]_t)z} - t e^{([d]_t -t[d-1]_t)z}}
\\ &=& \label{setq=1}
\frac{(1-t)e^{s^dz}}{e^{t^d z}-te^z}
\end{eqnarray}

It follows from  (\ref{formAn}) and Proposition~\ref{multinom} that
$$A_{dk}^{\maj,\exc,\fix}(\omega_d,t\omega_d^{-1},s) =   \sum_{m \ge 0}   \!\!\!\!\sum_{\scriptsize
\begin{array}{c} l_0\ge 0  \\ l_1,\dots, l_m \ge 1 \\ \sum l_i = k
\end{array}} \left(\begin{array}{c} k \\l_0,\dots,l_m\end{array}\right)\,\,s^{d l_0}
\prod_{i=1}^m t[dl_i-1]_{t}.$$
Hence, by straight-forward  manipulation of formal power series we have,
\begin{eqnarray*}& &\hspace{-.7in} \sum_{k\ge 0}  A_{dk}^{\maj,\exc,\fix}(\omega_d,t \omega_d^{-1},s) \frac{z^k}{k!} \\ & = & \sum_{k\ge 0} \sum_{m \ge 0}   \!\!\!\!\sum_{\scriptsize
\begin{array}{c} l_0\ge 0  \\ l_1,\dots, l_m \ge 1 \\ \sum l_i = k
\end{array}} \left(\begin{array}{c} k \\l_0,\dots,l_m\end{array}\right)\,\,s^{d l_0}
\prod_{i=1}^m t[dl_i-1]_{t} \frac {z^k} {k!}
\\ &=& e^{s^dz} \sum_{m,k \ge 0} \!\!  \!\!\!\!\sum_{\scriptsize
\begin{array}{c}\sum l_i =k \\ l_1,\dots, l_m \ge 1\end{array}} \left(\begin{array}{c} k \\l_1,\dots,l_m\end{array}\right)\,\,
\prod_{i=1}^m t[dl_i-1]_{t} \frac {z^k} {k!}
\\ &=& e^{s^dz} \sum_{m \ge 0} \sum_{ l_1,\dots, l_m \ge 1} \prod_{i=1}^m t[dl_i-1]_{t} \frac {z^{l_i} }{l_i!}
\\ &=& e^{s^dz} \sum_{m\ge 0} \left( \sum_{l\ge 1}  t[dl-1]_{t} \frac {z^{l} }{l!}\right)^m.
\end{eqnarray*}
Further manipulation yields,
\begin{eqnarray*}\sum_{m\ge 0} \left( \sum_{l\ge 1}  t[dl-1]_{t} \frac {z^{l} }{l!}\right)^m &=&
 \frac{ 1} {1- \left( \sum_{l\ge 1}  t[dl-1]_{t} \frac {z^{l} }{l!}\right)}
\\&=&
\frac{ 1-t }{1- t + \sum_{l\ge 1}  t(t^{dl-1}-1) \frac {z^{l} }{l!}}
\\&=&
\frac{ 1-t}{1- t + e^{t^dz}-1 -t(e^z -1)}
\\&=&
 \frac{1-t} {e^{t^dz} - t e^z}.
\end{eqnarray*}
Hence 
$$ \sum_{k\ge 0}  A_{dk}^{\maj,\exc,\fix}(\omega_d,t \omega_d^{-1},s) \frac{z^k}{k!}=  \frac{(1-t)e^{s^dz}} {e^{t^dz} - t e^z}.$$ The result now follows from  (\ref{setq=1}).
\end{proof}

\begin{cor} \label{cor3pol} Let $dk=n$.  Then 
$$A_n^{\maj,\exc}(\omega_d,t\omega_d^{-1})
 = A_k(t)[d]_t^k.$$
\end{cor} 

A similar result holds for the cycle-type $q$-Eulerian polynomials $A_{(n)}^{\maj,\exc} (q,t)$. 

\begin{thm} \label{circor} Let $dk = n-1$. Then $$A_{(n)}^{\maj,\exc}(\omega_d,t \omega_d^{-1}) =t A_{k}(t)\, [d]_t^k .$$
\end{thm} 

\begin{proof} We apply Proposition~\ref{evalth} which tells us that for all $j$, the coefficient of $t^j$ in $A_{(n)}^{\maj,\exc}(\omega_d,t \omega_d^{-1})$ is equal to $\chi^{Q_{(n),j}}_{1d^k}$. By Theorem~\ref{psum}, $\chi^{Q_{(n),j}}_{1d^k}$ equals the coefficient of $t^{j}$ in 
$G_{1d^k} (t)= t A_{k}(t) [d]_t^k$. \end{proof}

\begin{cor} \label{cycleth} Let $S_{n,j}$ be the set of permutations in $S_n$ with $j$ excedances.  Then  the triple $(G_{n}, S_{n,j}, a_{(n+1),j+1}(q))$ exhibits the cyclic sieving phenomenon for all $j \in \{0,1,\dots,n-1\}$.  
\end{cor}

\begin{proof}
That the triple exhibits the  cyclic sieving phenomenon is equivalent to the equation $$t\sum_{\sigma \in C_{S_{n}}(g)}t^{\exc(\sigma)} = A_{(n+1)}^{\maj,\exc}(\omega_d, t\omega_d^{-1}),$$ for all divisors $d$ of $n$ and  $g \in G_n$ of order $d$.
This equation  is a consequence of    Theorems~\ref{3pol} and~\ref{circor},  which respectively  say  that    the left side and the right side of the equation both equal $tA_k(t) [d]_t^{\frac n d}$.
\end{proof}

\end{document}